\documentclass[11pt]{amsart}

\usepackage{graphicx}
\usepackage{amssymb}
\usepackage[all]{xy}
\usepackage{color}

\usepackage{lineno}
\setpagewiselinenumbers

 \newtheorem{theorem}{Theorem}[section]
 \newtheorem{proposition}[theorem]{Proposition}
 \newtheorem{corollary}[theorem]{Corollary}
 \newtheorem{remark}[theorem]{Remark}

 \newtheorem{lemma}[theorem]{Lemma}
 \newtheorem{thmm}{Theorem}
 \newtheorem{corr}[thmm]{Corollary}

 \theoremstyle{definition}
 \newtheorem{example}[theorem]{Example}

 \newcommand{\Ker}{\mathop{\rm Ker}\nolimits}
 \newcommand{\im}{\mathop{\rm Im}\nolimits}
  \newcommand{\Hom}{\mathop{\rm Hom}\nolimits}

 \newcommand{\wX}{{\widetilde{X}}}
 
 \newcommand{\wY}{{\widetilde{Y}}}
 \newcommand{\wf}{{\tilde{f}}}

 \newcommand{\hX}{{\widehat{X}}}
 
 \newcommand{\hF}{{\widehat{F}}}
 \newcommand{\hH}{{\widehat{H}}}

 \newcommand{\HEG}{{\mathcal G}}
 \newcommand{\UU}{\mathcal{U}}

 \newcommand{\NN}{{\mathbb{N}}}
 \newcommand{\ZZ}{{\mathbb{Z}}}
 
 \newcommand{\RR}{{\mathbb{R}}}

\newcommand{\ignore}[1]{} 

\newcount\hh
\newcount\mm
\mm=\time
\hh=\time
\divide\hh by 60
\divide\mm by 60
\multiply\mm by 60
\mm=-\mm
\advance\mm by \time
\def\hhmm{\number\hh:\ifnum\mm<10{}0\fi\number\mm}

\begin{document}

\title[Uncountable Groups]{Uncountable groups and the geometry of inverse limits of coverings}

\author[G. Conner]{Gregory R. Conner$^1$}
\address{
Department of Mathematics,
Brigham Young University,
Provo, UT 84602, USA}
\email{conner@mathematics.byu.edu}

\author[W. Herfort]{Wolfgang Herfort$^2$}
\address{
Institute for Analysis and Scientific Computation
Technische Universit\"at Wien
Wiedner Hauptstra\ss e 8-10/101}
\email{wolfgang.herfort@tuwien.ac.at}

\author[C. Kent]{Curtis Kent$^3$}
\address{
Department of Mathematics,
Brigham Young University,
Provo, UT 84602, USA}
\email{curtkent@mathematics.byu.edu}

\author[P. Pave\v si\'c]{Petar Pave\v si\'c$^4$}
\address{
Faculty of Mathematics and Physics
University of Ljubljana
Jadranska 21
Ljubljana, Slovenia}
\email{petar.pavesic@fmf.uni-lj.si}

\thanks{
$^1$ Supported by Simons Foundation Collaboration Grant 646221.}
\thanks{$^2$ The second author is grateful for the warm hospitality at the Mathematics Department of BYU in February 2018.}
\thanks{$^3$ Supported by Simons Foundation Collaboration Grant 587001.}
\thanks{$^4$ Supported by the Slovenian Research Agency program P1-0292 and grants N1-0083, N1-0064}

\date{\today}

\begin{abstract}

In this paper we develop a new approach to the study of uncountable fundamental groups by using Hurewicz fibrations with the unique path-lifting property (\emph{lifting spaces} for short) as a replacement for
covering spaces.  In particular, we consider the inverse limit of a sequence of covering spaces of $X$.  It is known that the path-connectivity of the inverse limit can be expressed by means of the derived inverse limit functor $\varprojlim^1$, which is, however, notoriously difficult to compute when the $\pi_1(X)$ is uncountable.
To circumvent this difficulty, we express the set of path-components of the inverse limit, $\wX$, in terms of the functors $\varprojlim$ and $\varprojlim^1$ applied to sequences of countable groups arising from polyhedral approximations of $X$.

A consequence of our computation is that path-connectedness of lifting space implies that $\pi_1(\wX)$ supplements $\pi_1(X)$ in $\check\pi_1(X)$ where $\check\pi_1(X)$ is the inverse limit of fundamental groups of polyhedral approximations of $X$.  As an application we show that $\HEG\cdot \Ker_\ZZ(\hF)= \hF\ne\HEG\cdot \Ker_{B(1,n)}(\hF)$, where $\widehat F$ is the canonical inverse limit of finite rank free groups, $\HEG$ is the fundamental group of the Hawaiian Earring, and $\Ker_A(\hF)$ is the intersection of kernels of homomorphisms from $\widehat{F}$ to $A$.


\end{abstract}

\maketitle
\section{Introduction}

A famous theorem of Shelah \cite{Shelah} states that the fundamental
groups of Peano continua present a striking dichotomy: they are either
finitely presented or uncountable. The first case corresponds to
fundamental groups of finite polyhedra and covering space theory has traditionally been an
effective geometric approach to the study of the structure of these groups.
The second case is by no means exotic either: Peano continua with uncountable fundamental
group arise naturally as attractors of dynamical systems \cite{Hata85}, as fractal spaces
\cite{Mandelbrot85,Massopust89}, as boundaries of non-positively curved groups
\cite{KapovichKleiner00}, and in many other situations. The fundamental difference between
the two cases is of a local nature.
If the fundamental group of a Peano continuum $X$ is uncountable, then by \cite{Shelah} it is
not semilocally simply connected at some point. As a consequence most of
the subgroups of $\pi_1(X)$ do not correspond to a covering space over $X$, which represents a major obstacle
for a geometric study of these groups. This work is part of a wider program to study fundamental
groups of `wild' Peano continua where the role of covering spaces is taken by more general fibrations with
the unique path-lifting
property. These fibrations were introduced by Spanier \cite[Chapter 2]{Spanier} who developed much
of the theory of covering spaces within this more general setting. Since the term `Hurewicz fibration
with the unique path-lifting property' is somewhat impractical we call them \emph{lifting spaces}
and the corresponding maps from the total space to the base are called \emph{lifting projections}.

The main advantage of lifting spaces over covering spaces is that the former are closed with respect to
composition and arbitrary inverse limits - see \cite[Section II.2]{Spanier}. Most notably, the inverse limit of
a sequence of covering spaces over $X$ is always a lifting space over $X$ (and is not a covering
projection, unless the sequence is eventually constant). If $X$ is semilocally simply-connected (e.g.,
a CW-complex), and the sequence is not eventually constant, then the limit lifting space is path-disconnected
(see Corollary \ref{cor:cov over CW}).

This is a geometric reflection of the fact that the fundamental group of the base is
countable while the fibre of the lifting projection is uncountable. However, if $X$ is not semilocally
simply-connected, then its fundamental group is uncountable by Shelah's theorem and the limit space
can be path-connected or not path-connected, depending on the interplay between $\pi_1(X)$ and the sequence
of subgroups corresponding to the coverings in the inverse sequence.

The path-components of an inverse limit of covering spaces are classically determined by the derived inverse limit functor $\varprojlim^1$ applied to a sequence of subgroups of $\pi_1(X)$.  Difficulties arise in the computation of $\varprojlim^1$ for inverse sequences of uncountable groups, which make this an ineffective approach to determining path-connectivity
(see discussion at the end of Section \ref{sec:Preliminaries on inverse limits of spaces and groups}
and Examples 3.7-3.11 in Section \ref{sec: inverse limits of coverings})
.  In Section 3, we consider inverse sequences of coverings over some polyhedral expansion of the base space $X$ and study the path-connectedness of the limit.  We prove the following result, which allows one to work with inverse sequences of countable groups and thus avoid the difficulties of computing $\varprojlim^1$ for uncountable groups.

\begin{thmm}[Theorem \ref{thm:limit over PC}]
Every inverse limit of covering projections over a Peano continuum $X$ is homeomorphic to an inverse limit of covering projections over a polyhedral expansion of $X$.\end{thmm}

The main result of Section 3, Theorem \ref{thm:expansion lim1}, completely describes
the fundamental group and the set of path components of an inverse limit of coverings over
some polyhedral expansions of a Peano continuum. The statement is quite technical, but it leads to
the following important consequence which characterizes path-connectedness of the limit space.

\begin{corr}[Corollary \ref{cor:expansion 0-conn}]
Let $p\colon \wX\to X$ be the inverse limit of covering maps over a Peano continuum $X$. Then $\wX$ is path-connected if, only if, the inverse sequence of fundamental groups, $\pi_1(\wX_i)$,
satisfies the Mittag-Leffler property and the natural homomorphism
$$\pi_1(X)\longrightarrow \varprojlim
\left(\frac{\pi_1(X_i)}{\pi_1(\wX_i)}\right)$$
is surjective, where $\wX_i\to X_i$ are covering projections of a polyhedral expansion $\{X_i\}$ of $X$.\end{corr}

Even for the well-studied Hawaiian earring group $\HEG$, apart from the fact that $\HEG$ contains the free group of countable rank and is thus dense in its shape group $\hF$, very little is known about the size of $\HEG$ in $\hF$. The above corollary leads to interesting algebraic applications concerning the image of the fundamental group into the shape group which we consider in Section 4.

Given groups $G$ and $A$, we let $\Ker_A(G)$ be the intersection of kernels of all homomorphisms from $G$ to $A$. If $G$ is a free group and $B(1,n)$ is a Baumslag-Solitar group, then $\Ker_\ZZ(G)$ is exactly the commutator subgroup of $G$ and $\Ker_{B(1,n)}(G)$ is exactly the second derived subgroup of $G$, see \cite{Conner-Kent-Herfort-Pavesic}. However, $\hF$ is locally free but non-free and the corresponding $\Ker_\ZZ(\hF)$ properly contains the commutator subgroup of $\hF$. In fact, $\Ker_\ZZ(\hF)$ is sufficiently large to be a supplement of $\HEG$ in $\hF$ while $\Ker_{B(1,n)}(\hF)$ is not.

\begin{thmm}[Theorem \ref{thm:HEG internal} and Theorem \ref{thm:HEG}]
The group $\hF$ is equal to the internal product of its subgroups
$\HEG$ and $\Ker_\ZZ(\hF)$, while the internal product of
$\HEG$ and $\Ker_{B(1,n)}(\hF)$ is a proper subgroup of $\hF$.\end{thmm}


\section{Preliminaries on inverse limits of spaces and groups}
\label{sec:Preliminaries on inverse limits of spaces and groups}

The main object of our study are lifting spaces that arise as inverse limits of covering spaces,
so let us consider the following sequence of regular covering maps
$p_i\colon \wX_i\to X$, together with the inverse limit map $p\colon \wX\to X$:
$$\xymatrix{
\wX_1 \ar[d]_{p_1} & \wX_2 \ar[d]_{p_2} \ar[l]_{\wf_{1}} & \wX_3 \ar[d]_{p_3} \ar[l]_{\wf_{2}} & \cdots\cdot\ar[l] & \wX \ar[l]\ar[d]^p\\
X \ar@{=}[r] & X \ar@{=}[r] & X \ar@{=}[r] & \cdots\cdots\ar@{=}[r] & X}$$
For each $i$ we identify $\pi_1(\wX_i)$ with its image $p_{i*}(\pi_1(\wX_i))$ which is a normal
subgroup of $\pi_1(X)$. By standard covering space theory the
bonding maps $\wf_{i}: \wX_{i+1}\to \wX_i$ are covering maps and we have
a decreasing sequence of normal subgroups of $\pi_1(X)$
$$\pi_1(X) \unrhd \pi_1(\wX_1) \unrhd \pi_1(\wX_2) \unrhd\cdots $$

Note that the converse is not true for general spaces, because a subgroup of $\pi_1(X)$ may not be a covering
subgroup, i.e. it does not necessarily come from a covering of $X$. However, if $X$ is locally path-connected, then
by \cite[Theorem II, 5.12]{Spanier} every decreasing sequence of normal covering subgroups of $\pi_1(X)$ uniquely
determines an inverse sequence of regular coverings over $X$.

Let $F(p)$ denote the fibre of the lifting projection $p$. Then $F(p)$  is the inverse limit of the fibres
$F(p_i)$, and so by \cite[Theorem II, 6.2]{Spanier} it can be naturally identified with the inverse limit of
the quotients
$$F(p)=\varprojlim (\pi_1(X)/\pi_1(\wX_i)).$$

Since all bonding maps in the inverse sequence are fibrations, it is well-known (see \cite{Cohen70, Hirschorn15})
that the homotopy groups of $\pi_1(\wX)$
can be expressed in terms of the homotopy groups of $\wX_i$:
$$\pi_1(\wX)=\varprojlim \pi_1(\wX_i)=\bigcap_i\pi_1(\wX_i),\ \ \ \ \pi_0(\wX)=\varprojlim{^1} \pi_1(\wX_i)$$
and $\pi_n(\wX)=\pi_n(X)$ for $n>1$. Here $\varprojlim$ denotes the inverse limit functor on groups
and $\varprojlim^1$ is its first derived functor
(see \cite[Section 11.3]{Geoghegan}).
For commutative groups these functors can be defined as kernel and cokernel of the homomorphism
$\varphi\colon \prod \pi_1(\wX_i)\to \prod \pi_1(\wX_i)$, given as
$$\varphi(u_1,u_2,u_3,\ldots)=(u_1-\wf_{1*}(u_2),u_2-\wf_{2*}(u_3),\ldots)$$
and so they fit in the exact sequence
$$0\to  \varprojlim \pi_1(\wX_i) \to \prod \pi_1(\wX_i) \stackrel{\varphi}{\longrightarrow}
\prod \pi_1(\wX_i) \to \varprojlim{^1}  \pi_1(\wX_i)\to 0.$$
The non-commutative case is more delicate: in that case $\varprojlim{^1} \pi_1(\wX_i)$
is defined as the quotient of $\prod \pi_1(\wX_i)$ under the equivalence relation given as
$$(u_i)\sim (v_i) \Leftrightarrow (v_i)=(x_i u_i\wf_{i*}(x_{i+1})^{-1})
\text{ for some } (x_i)\in \prod \pi_1(\wX_i)$$

(see \cite{Geoghegan1980} for detailed treatment of the non-commutative
case).

The values of $\varprojlim^1$ are notoriously hard to compute. Here we will be only interested
whether $\varprojlim^1$ of a sequence is trivial, and this can be settled if we can show
that the inverse sequence satisfies the \emph{Mittag-Leffler condition}, which we now define.
In an inverse sequence of groups
$$G_1\longleftarrow G_2 \longleftarrow G_3\longleftarrow\cdots$$
for a fixed $j$ the image of the homomorphism $G_i\to G_j$ decreases as $i$ goes toward infinity. The inverse
sequence is said to satisfy the \emph{Mittag-Leffler condition} if for every $j$ the sequence
$\{\im(G_i\to G_j)\mid i=j,j+1,\ldots\}$ stabilizes. Clearly, sequences with epimorphic bonding maps satisfy
the Mittag-Leffler condition. On the other hand, if the bonding maps
are monomorphisms, then the sequence satisfies he Mittag-Leffler condition if, and only if, it is eventually
constant. The following result is proved in \cite{Geoghegan1980} (see also \cite[Theorem 11.3.2]{Geoghegan}):

\begin{proposition}
\label{prop:lim1 and ML}
If an inverse sequence $\{G_i\}$, of groups, satisfies the Mittag-Leffler
condition, then $\varprojlim^1 G_i=\{1\}$. Conversely, if
$\varprojlim^1 G_i=\{1\}$ and each $G_i$ is countable, then
$\{G_i\}$ satisfies the Mittag-Leffler condition.
\end{proposition}

As we explained before, the sequence of fundamental groups induced by an inverse sequence of coverings
satisfies Mittag-Leffler condition only if it is constant from some point on. Thus we get immediately
the following corollary.

\begin{corollary}
\label{cor:cov over CW}
Let $\wX$ be the inverse limit of an inverse sequence of covering maps $p_i\colon \wX_i\to X$ and assume that
for some $i$ the group $\pi_1(\wX_i)$ is countable.
Then either the sequence is eventually constant (which implies that $p\colon\wX\to X$ is a covering map),
or $\wX$ is not path-connected.

In particular an inverse limit of coverings over a countable CW-complex is either a covering or its
total space is not path-connected.
\end{corollary}

This last statement is not surprising if one considers the tail of the exact homotopy sequence of the
fibration $p\colon \wX\to X$, shown in the following diagram together with the above mentioned identifications:
$$\xymatrix@C=1.5pc{
1\ar[r] & \pi_1(\wX) \ar[r]^{p_*}\ar@{=}[d] & \pi_1(X) \ar[r]^\partial\ar@{=}[d] & \pi_0(F(p)) \ar[r]\ar@{=}[d]
& \pi_0(\wX) \ar[r]\ar@{=}[d] & {*} \\
1\ar[r] & \varprojlim\pi_1(\wX_i) \ar[r] & \pi_1(X) \ar[r] & \varprojlim (\pi_1(X)/\pi_1(\wX_i)) \ar[r] &
\lim^1\pi_1(\wX_i) \ar[r] & {*} }
$$
If the sequence is not eventually constant, then $\pi_0(F(p))$ is uncountable.
But if $X$ is a countable CW-complex, then $\pi_1(X)$ is a countable group and so $\partial$
cannot be surjective.

However, if $\pi_1(X)$ is uncountable, then $\partial$ can be surjective. For example, the lifting projection
over the infinite product of circles, obtained by 'unwrapping' one circle at a time, as in
$$
(S^1)^\infty \leftarrow  (\RR\times S^1\times S^1\times\ldots) \leftarrow (\RR\times \RR\times S^1\times\ldots)
\leftarrow \cdots\cdots \leftarrow \RR^\infty
$$
has a path-connected total space in spite of the fact that it is not eventually constant. Another, less
obvious example is a sequence of 2-fold coverings over the Hawaiian earring whose limit space is path-connected
(see \cite[Section 2.7]{Conner-Herfort-Pavesic}). In both cases we have decreasing sequences of
uncountable fundamental groups that do not satisfy the Mittag-Leffler condition
but their derived inverse limits are nonetheless trivial.

Of course, one can easily find examples of inverse limits of coverings of the Hawaiian ring whose total space
has uncountably many path-components. Distinguishing between different cases is hindered by the difficulties
in the explicit computation of the derived limit functor for decreasing sequences of uncountable groups.
We present one approach to this problem in the following section, while in the last section we
consider the algebraic implications of the path-(dis)connectedness of the limit space.

\section{Inverse limits of coverings over an expansion of $X$}\label{sec: inverse limits of coverings}

Our approach to reducing the complexity of computing the derived limit functor to the more manageable countable case is to consider a polyhedral expansion $\{X_i\}$ of $X$ and represent a lifting projection $p\colon\wX \to X$ as an inverse limit of covering maps over the polyhedral expansion, i.e., find covering maps $p_i\colon\wX_i\to X_i$ that satisfy the following diagram.

$$\xymatrix{
\wX_1 \ar[d]_{p_1} & \wX_2 \ar[d]_{p_2} \ar[l]_{\wf_{1}} & \wX_3
\ar[d]_{p_3} \ar[l]_{\wf_{2}} & \cdots\ar[l] & \wX \ar[d]^p\ar[l]\\
X_1  & X_2 \ar[l]^{f_{1}} & X_3 \ar[l]^{f_{1}} & \cdots\ar[l] & X\ar[l]}$$

There are several standard methods by which a compact metric space can be represented as a limit of an inverse sequence of polyhedra (see Mardesic-Segal \cite{Mardesic-Segal}). The one that best  suits our purposes is by nerves of coverings, which we briefly recall (see \cite[Appendix 1]{Mardesic-Segal} for details). Every metric compactum $X$ admits arbitrarily
fine finite open coverings. Moreover, for every finite open covering $\UU=\{U_1,\ldots,U_n\}$ of $X$ there exists a subordinated partition of unity $\{\rho_i\colon X\to [0,1]\mid i=1,\ldots,n\}$. Let $N(\UU)$ be the \emph{nerve} of the covering $\UU$, i.e., the simplicial complex whose vertices are elements of $\UU$, and whose simplices are spanned by elements of $\UU$ with non-empty intersection. Then the formula
$$f(x):=\sum_{i=1}^n \rho_i(x)\cdot U_i$$
defines a map $f\colon X\to |N(\UU)|$ (where $|N(\UU)|$ is the geometric realization of the nerve). It is well known that the choice of the partition of unity does not affect the homotopy class of $f$, so the induced homomorphism $f_*\colon \pi_1(X)\to\pi_1\bigl(|N(\UU)|\bigr)$ depends only on the cover $\UU$.

\begin{lemma}
\label{lem:epi mf}
If each element of $\UU$ is path-connected, then $f_*\colon \pi_1(X)\to\pi_1\bigl(|N(\UU)|\bigr)$ is surjective.
\end{lemma}

\begin{proof}
Without loss of generality we may assume that the partition of unity
subordinated to $\UU$ is reduced in the sense that for every $i$
there exists $x_i\in U_i$, such that $\rho_i(x_i)=1$.
For every pair of intersecting sets $U_i, U_j\in\UU$ we
may choose a path in $U_i\cup V_j$ between $x_i$ and $x_j$.
These paths determine a map $g\colon |N(\UU)^{(1)}|\to X$ from
the 1-skeleton of the nerve to $X$. One can check that, for every 1-simplex $\sigma$ in $N(\UU)^{(1)}$, the image
$f(g(\sigma))$
is contained in the open star of $\sigma$ in $N(\UU)$, which implies
that $f\circ g$ is homotopic to the inclusion
$i\colon|N(\UU)^{(1)}| \hookrightarrow |N(\UU)|$. Since
$i_*\colon\pi_1(|N(\UU)^{(1)}|)\to\pi_1(|N(\UU)|)$ is surjective,
 $f_*$ is also surjective.
\end{proof}

If $\UU'$ is a covering of $X$ that refines $\UU$ (i.e. every element
of $\UU'$ is contained in some element of $\UU$), then there is an
obvious simplicial map $N(\UU')\to N(\UU)$. By iterating the
refinements we obtain an inverse system of polyhedra
$$|N(\UU_1))| \longleftarrow |N(\UU_2))| \longleftarrow |N(\UU_3))|
\longleftarrow\ldots$$
whose limit is $X$. Moreover, if $X$ is locally path-connected, then
we may choose covers of $X$ whose elements are path-connected. As an
interesting aside, Lemma \ref{lem:epi mf} immediately implies the following
well-known fact, which also follows immediately from the Hahn-Mazurkiewicz theorem and results of Krasinkiewicz \cite[Theorems 4.1 and
4.2]{Krasinkiewicz}.

\begin{corollary}
Every Peano continuum $X$ can be represented as the limit of
an inverse system of finite polyhedra
$$\xymatrix{
X_1 & X_2 \ar[l] & X_3 \ar[l] & \cdots \ar[l] &\ar[l] \varprojlim X_i= X
}$$
such that the homomorphisms
$\pi_1(X)\to\pi_1(X_i)$  induced by the projection maps are surjective.
\end{corollary}

We do not know whether every inverse limit of coverings over a
path-connected and locally path-connected
base $X$ can be derived from an inverse limit of coverings over
some expansion of $X$. However, if we assume that $X$ is compact,
we may use the approximation by nerves of coverings to prove the following theorem.

\begin{lemma}\label{lem:kernel}Let $X$ be a Peano continuum and $g: X \to Y$ be a continuous map into a polyhedron $Y$.  If $\mathcal U$ is a cover of $X$ by path connected open sets such that the preimage of the open star of any vertex of $Y$ is contained in an element of $\mathcal U$, then $\Ker(g_*)\subset \pi_1\bigl( X; 2\,\mathcal U\bigr)$.
\end{lemma}

\begin{proof}
  If we denote by $f\colon X\to |\mathcal{U}|$ the map from $X$ to the realization of the nerve of $\mathcal{U}$, then the kernel of the homomorphism $f_*\colon\pi_1(X)\to\pi_1(|\mathcal{U}|)$ is $\pi_1(X;2\,\mathcal{U})$, see \cite[Lemma 3.3]{Conner-Pavesic}.  Since the preimage of the open star of any vertex of $Y$ is contained in an element of $\mathcal U$, there exists a map $h: Y \to |\mathcal{U}|$ such that $f_* = h_*\circ g_*$.  Thus $\Ker(g_*)\subset \ker(f_*) = \pi_1\bigl( X; 2\,\mathcal U\bigr)$.
\end{proof}

Let $\mathcal U$ be an open cover of $X$.  We will use $\pi_1(X; \mathcal U)$ to denote the $\mathcal U$-Spanier subgroup, i.e. $\pi_1(X; \mathcal U)$ is the subgroup of $\pi_1(X,x_0)$ generated by $$\bigl\{[\alpha * \beta *\overline \alpha] \mid \alpha: (I,0) \to (X, x_0) \text{ and } \im(\beta)\subset U \text{ for some } U\in \mathcal U \bigr\}.$$
\begin{theorem}
\label{thm:limit over PC}
If $X$ is a Peano continuum, then every inverse sequence of covering
projections over $X$ is homeomorphic to
a pull-back of an inverse sequence of covering maps over a given
polyhedral expansion of $X$.

As a consequence, every inverse limit of covering maps over $X$ is homeomorphic to an inverse limit
of covering maps over a given polyhedral expansion of $X$.
\end{theorem}


\begin{proof}
Let
$$\xymatrix{
X_1 & X_2 \ar[l] & X_3 \ar[l] & \cdots \ar[l] & \varprojlim X_i = X\ar[l]
}$$
be a polyhedral expansion of $X$, such that the homomorphisms
$\pi_1(X)\to\pi_1(X_i)$ are surjective.

Let $p\colon\wX\to X$ be an arbitrary covering map. Then $p$ evenly covers elements of some open cover $\mathcal{U}$ of $X$
and $\pi_1(\wX)\supseteq \pi_1(X;\mathcal{U})$, where the latter denotes the $\mathcal U$-Spanier subgroup of $\pi_1(X)$. Since $X$ is compact we may assume that $\mathcal{U}$ is finite.
Choose another finite cover $\mathcal{V}$ of $X$ by path connected open sets whose double $2\,\mathcal{V}$ (set of all unions of two elements of $\mathcal{V}$) refines $\mathcal{U}$.


By Theorem 5 of \cite{Mardesic-Segal} there exists an $i$ such that the preimage of the open star of any vertex of $X_i$ is contained in an element of $\mathcal V$. Then by Lemma \ref{lem:kernel}, we have
$$\pi_1(\wX)\supseteq\pi_1(X;\mathcal{U})\supseteq\pi_1(X;2\,\mathcal{V})\supseteq \Ker\big(\pi_1(X)\to\pi_1(X_i)\big).$$
Therefore, if
$$\xymatrix{
X & \wX_1 \ar[l] & \wX_2 \ar[l] & \wX_3 \ar[l] & \cdots \ar[l]
}$$
is an inverse sequence of coverings over $X$ we may find for each $i$ some index $j(i)$ such that
$\pi_1(\wX_i)\supseteq \Ker\big(\pi_1(X)\to\pi_1(X_{j(i)})\big)$.

Denote $K_i:=\Ker\big(\pi_1(X)\to\pi_1(X_{j(i)})\big)$ and let $\overline X_{j(i)}\to X_{j(i)}$ be
the covering map that corresponds to the subgroup $\im \big(\pi_1(\wX_i)\to\pi_1(X_{j(i)})\big)$.
By the lifting theorem, there is a fibre-preserving map $\wf_i\colon \wX_i\to \overline X_{j(i)}$ for which
the following diagram commutes
$$\xymatrix{
\wX_i \ar[d] \ar[r]^{\wf_i} & \overline X_{j(i)} \ar[d]\\
X \ar[r] & X_{j(i)}}
$$
To compute the restriction of $\wf_i$ on the fibres we examine the following commutative diagram with exact rows
$$\xymatrix{
1 \ar[r] & \pi_1(\wX_i)/K_i \ar[r]\ar@{->>}[d] & \pi_1(X)/K_i\ar[r]\ar[d]^\cong &  \pi_1(X)/\pi_1(\wX_i) \ar[r]\ar[d] & 1\\
1 \ar[r] & \pi_1(\overline X_{j(i)}) \ar[r] & \pi_1(X_{j(i)})\ar[r] &  \pi_1(X_{j(i)})/\pi_1(\overline X_{j(i)}) \ar[r] & 1
}$$
By the short five-lemma we conclude that $\wf_i$ induces an isomorphism
$$\pi_1(X)/\pi_1(\wX_i)\to \pi_1(X_{j(i)})/\pi_1(\overline X_{j(i)}),$$
and hence a  bijection between the fibres of respective covering maps.
In other words, the above diagram is a pull-back of covering maps.

By repeating the above argument for all terms in an inverse sequence of covering maps over $X$ we obtain
 an inverse sequence of covering maps over the given expansion of $X$ with the same limit lifting projection.
\end{proof}

%
%

The above theorem shows that for most cases of interest it is sufficient to consider lifting projections over $X$ that
are inverse limits of covering maps over some expansion of $X$ as in

$$\xymatrix{
\wX_1 \ar[d]_{p_1} & \wX_2 \ar[d]_{p_2} \ar[l]_{\wf_{1}} & \wX_3 \ar[d]_{p_3} \ar[l]_{\wf_{2}} & \cdots\cdot\ar[l] & \wX \ar[l]\ar[d]^p\\
X_1  & X_2 \ar[l]^{f_{1}} & X_3 \ar[l]^{f_{2}} & \cdots\cdots\ar[l] & X\ar[l]}$$
Suppose  that $X_i$ are polyhedra and that the maps $X\to X_i$ induce epimorphisms between
respective fundamental groups (e.g. if $X$ is locally path-connected).
As in Section 2, we may compare the tail of the exact
homotopy sequence of
the fibration $p$ with the derived exact sequence of the inverse
limit functor to obtain the following diagram
$$\xymatrix@C=1.5pc{
1\ar[r] & \pi_1(\wX) \ar[r]\ar[d]_\varphi & \pi_1(X) \ar[r]\ar[d]_\iota & \varprojlim (\pi_1(X_i)/\pi_1(\wX_i)) \ar[r]\ar@{=}[d] & \pi_0(\wX) \ar[r]\ar[d]^\psi & {*}\\
1\ar[r] & \varprojlim\pi_1(\wX_i) \ar[r] & \varprojlim\pi_1(X_i) \ar[r] & \varprojlim (\pi_1(X_i)/\pi_1(\wX_i)) \ar[r] & \varprojlim^1\pi_1(\wX_i) \ar[r] & {*}}
$$
 The inverse limits $\varprojlim\pi_1(X_i)$ and $\varprojlim\pi_1(\wX_i)$
are independent of the choice of a polyhedral expansion for $X$.  That two polyhedral expansions give pro-isomorphic inverse limit groups is proved in  Mardesic-Segal \cite[Ch 2]{Mardesic-Segal} and the pro-isomorphism can be lifted to a pro-isomorphism of the corresponding covering subgroups via the homotopy lifting criterion.
The limit $\varprojlim\pi_1(X_i)$ is usually called the
\emph{\v Cech fundamental group}  (or the \emph{first shape group})
of $X$, and is denoted by $\check\pi_1(X)$.   Since $\wX$ is not necessarily compact, $\varprojlim\pi_1(\wX_i)$ is not necessarily the first shape group of $\wX$.  Regardless, we will still denote  $\varprojlim\pi_1(\wX_i)$ by $\check\pi_1(\wX)$. The kernel of the induced map $\iota\colon\pi_1(X)\to \check\pi_1(X)$
is called the \emph{shape kernel} of $X$ and denoted $\mathrm{ShKer}(X)$.
Note that by the exactness of the above diagram we have that
$$\Ker(\pi_1(\wX)\to\varprojlim\pi_1(\wX_i))=
\Ker(\pi_1(X)\to\varprojlim \pi_1(X_i))=\mathrm{ShKer}(X).$$

Clearly, $\psi\colon\pi_0(\wX)\to \varprojlim^1\pi_1(\wX_i)$ is a
surjective function
between homogeneous sets. Thus, in order to determine
$\pi_0(\wX)$ we need to
compute the fibres of $\psi$. From the exactness of the second row we
deduce that $\varprojlim^1\pi_1(\wX_i)$ is the set of cosets
of the action of the group $\check\pi_1(X)/\check\pi_1(\wX)$
on $\varprojlim (\pi_1(X_i)/\pi_1(\wX_i))$. Then a straightforward
diagram chasing shows that the fibres of $\psi$ can be naturally
identified with the cosets of the action of $\iota(\pi_1(X))$ on
$\check\pi_1(X)/\check\pi_1(\wX)$.

\begin{theorem}
\label{thm:expansion lim1}
Let $X$ be a path-connected space that admits a polyhedral expansion
$f_i\colon X\to X_i$ such that the induced homomorphisms $f_{i*}$ are surjective and
let $p\colon\wX\to X$ be the inverse limit of a sequence of coverings
$p_i\colon\wX_i\to X_i$ as described above. Then the fundamental group of $\wX$ is determined by the exact sequence of groups
$$1\to\mathrm{ShKer}(X)\to\pi_1(\wX)\to\check \pi_1(\wX),$$
and the set of path-components of $\wX$ is determined by the exact sequence of
groups and based homogeneous sets
$$\pi_1(X)\stackrel{\iota}{\longrightarrow} \check\pi_1(X)/\check\pi_1(\wX)
\longrightarrow \pi_0(\wX) \longrightarrow \varprojlim{\!^1} \pi_1(\wX_i) \longrightarrow *$$
\end{theorem}

The main advantage of Theorem \ref{thm:expansion lim1} with respect to the description of $\pi_0(\wX)$
given in Section \ref{sec:Preliminaries on inverse limits of spaces and groups} is that all limits and derived
limits are taken over inverse sequences of countable groups.

\begin{corollary}
\label{cor:expansion 0-conn}
Let $p\colon \wX\to X$ be the inverse limit of covering maps
as in the above
theorem. Then $\wX$ is path-connected if, only if, the inverse sequence
$$\pi_1(\wX_1)\leftarrow \pi_1(\wX_2)\leftarrow \pi_1(\wX_3)\leftarrow\cdots$$
satisfies the Mittag-Leffler property and the natural homomorphism
$$\pi_1(X)\longrightarrow \frac{\check\pi_1(X)}{\check \pi_1(\wX)}\cong\varprojlim
\left(\frac{\pi_1(X_i)}{\pi_1(\wX_i)}\right)$$
is surjective.
\end{corollary}
\begin{proof}
Theorem \ref{thm:expansion lim1} implies that $\pi_0(\wX)$ is trivial if, and only if
$\varprojlim{\!^1} \pi_1(\wX_i)$ is trivial and
$\pi_1(X)\longrightarrow \check\pi_1(X)/\check \pi_1(\wX)$ is surjective.
Since the fundamental groups of polyhedra are countable, then by Proposition \ref{prop:lim1 and ML} the inverse sequence is Mittag-Leffler if and only if $\lim{^1} \pi_1(\wX_i)$ is trivial.

Observe that the isomorphism between $\check\pi_1(X)/\check \pi_1(\wX)$ and
$\varprojlim \big(\pi_1(X_i)/\pi_1(\wX_i)\big)$ follows form the Mittag-Leffler property and does
not hold in general.
\end{proof}

Let us consider a couple of examples.

\begin{example}
\label{ex:torus covering}
We have already mentioned the inverse sequence of covering maps over the infinite product of circles
$$(S^1)^\infty\leftarrow (\RR\times S^1\times S^1\times\ldots) \leftarrow (\RR\times \RR\times S^1\times\ldots)
\leftarrow  \cdots $$
It can be replaced by the following inverse sequence of covering maps over polyhedral approximations
of $(S^1)^\infty$ (with horizontal maps the usual projections)
$$\xymatrix{
\RR \ar[d] & \RR\times\RR  \ar[l]\ar[d] & \RR\times\RR\times\RR \ar[l]\ar[d] & \cdots \ar[l] & \RR^{\NN}\ar[d]\ar[l]\\
S^1        & S^1\times S^1 \ar[l]       & S^1\times S^1\times S^1 \ar[l]     & \cdots \ar[l] & (S^1)^{\NN}\ar[l]
}$$
The fundamental groups of the products of copies of $\RR$ are trivial so the corresponding sequence is Mittag-Leffler.
By Corollary \ref{cor:expansion 0-conn} the path-connectedness of $\RR^{\NN}$ is equivalent to the equality
$\pi_1\bigl((S^1)^{\NN}\bigr)=\varprojlim \pi_1\bigl((S^1)^n\bigr)=\check\pi_1\bigl((S^1)^{\NN}\bigr)$.
\end{example}
\begin{example}
Let $T^\infty$ be the infinite product of circles endowed with the CW-topology, i.e. the direct limit of the
sequence of spaces
$$S^1\hookrightarrow S^1\times S^1\hookrightarrow S^1\times S^1\times S^1 \hookrightarrow \cdots T^\infty$$
Clearly, the fundamental group of $T^\infty$ is $\bigoplus_{k=1}^\infty\ZZ$.
For each $i\in\NN$ let $p_i\colon\wX_i\to T^\infty$ be the covering map that corresponds to the subgroup
$\pi_1(\wX_i)=\bigoplus_{k=i}^\infty\ZZ$. By Corollary \ref{cor:cov over CW} the inverse limit
$\wX:=\varprojlim \wX_i$ is not path-connected.
\end{example}

\begin{remark} The last result is somewhat counter-intuitive, as one
could argue that $\wX_i\approx \RR^{i-1}\times T^\infty$
and that
the inverse limit of the coverings is simply the product $\RR^\infty$
but the computation reveals that the geometry of the inverse limit
must be different. We leave this as an (easy) exercise for the reader.
As a hint, note that the fibre of $\wX\to T^\infty$ is uncountable while
the fundamental group of $T^\infty$ is countable.
\end{remark}

\begin{example}
Consider the squaring lifting projection $p\colon \widetilde H_\infty\to H$  over the Hawaiian earring $H$
described in \cite[Section 2.7]{Conner-Herfort-Pavesic}.
The map $p$ is obtained as a limit of a sequence of 2-fold covering
projections and, although it resembles at first sight the construction
of the dyadic
solenoid, we have been able to prove by a geometric argument that its
total space $\widetilde H_\infty$ is path-connected. We are going to show how this
result is reflected in the algebraic computation of the set of
path-components of $\widetilde H_\infty$.

As explained in \cite[Section 3.3]{Conner-Herfort-Pavesic} we can
obtain the projection $p$ as follows.
Let $H_i$ be a wedge of $i$ circles, so that
$H_1\leftarrow H_2\,\leftarrow \ldots \leftarrow H$ is the standard polyhedral
expansion of the Hawaiian earring. For each $i$ let
$p_i\colon \widetilde H_i\to H_i$ be the $2^i$-fold covering
projection, obtained as
the restriction to $H_i\subset (S^1)^i$ of the squaring map
$$(S^1)^i\to (S^1)^i,\ \ \ \ (z_1,\ldots,z_i)\mapsto (z^2_1,\ldots,z^2_i).$$
It is then easy to check that $p$ is the inverse limit of the sequence
of coverings $p_i$, in particular
$\widetilde H_\infty=\varprojlim \widetilde H_i$.

The sequence of groups $\bigl(\pi_1(\widetilde H_i)\bigr)$ has surjective bonding
maps, so by Corollary \ref{cor:expansion 0-conn}
the path-connectedness of $\widetilde H_\infty$ is equivalent to the surjectivity
of the homomorphism $\pi_1(H)\to (\ZZ_2)^\NN$. But the latter is
obvious, because any sequence $(a_i)\in (\ZZ_2)^\NN$, where
$a_i\in\{0,1\}$ can be obtained as the image of the loop that
winds around the $i$-th circle exactly when $a_i=1$.
\end{example}

\begin{example}
We are going to show that a minor modification of covering maps
in the inverse sequence described in example \ref{ex:torus covering}
yields a completely different limit space. This is surprising and
very difficult to see geometrically. Let us define two covering
projections with fibre $\ZZ$. The first is the standard covering
exponential map from the real line to the circle
$$e\colon \RR\to S^1, \ \ \ \ e(t):=e^{2\pi it} $$
and the second combines the exponential covering with the two-fold
covering of the circle given by the squaring map
$$f\colon \RR\times S^1\to S^1\times S^1,\ \ \ f(t,z):=(e^{2\pi it}+z^2,z),$$
Note that the induced homomorphism $f_*\colon \pi_1(\RR\times S^1)\to
\pi_1(S^1\times S^1)$ sends the generator of
$\pi_1(\RR\times S^1)\cong\ZZ$
to the element $(2,1)\in\pi_1(S^1\times S^1)\cong\ZZ\oplus\ZZ$,
 therefore $f$ can be characterized as the
covering of $S^1\times S^1$ corresponding to the cyclic subgroup of
$\ZZ\times \ZZ$ generated by the element $(2,1)$.

The following diagram depicts two inverse sequences of covering maps over $(S^1)^\NN$.
$$\xymatrix{
\vdots \ar[d] & & \vdots \ar[d]\\
\RR\times \RR\times \RR\times S^1\times\cdots \ar[d]^{1\times 1\times e\times 1\times\ldots}& &
\RR\times \RR\times (\RR\times S^1)\times\cdots \ar[d]^{1\times 1\times f\times\ldots}\\
\RR\times \RR\times S^1\times S^1\times\cdots \ar[d]^{1\times e\times 1\times 1\times \ldots}& &
\RR\times (\RR\times S^1)\times S^1\times\cdots \ar[d]^{1\times f\times 1\times\ldots}\\
\RR\times S^1\times S^1\times S^1\times\cdots \ar[d]^{e\times 1\times 1\times 1\times\ldots} & &
(\RR\times S^1)\times S^1\times S^1\times\cdots \ar[d]^{f\times 1\times 1\times\ldots}\\
S^1\times S^1\times S^1\times S^1\times\cdots & & S^1\times S^1\times S^1\times S^1\times\cdots
}$$
The first sequence was already considered in Example \ref{ex:torus covering} where we showed that
its inverse limit is $\RR^\NN$, which is, of course, path-connected. The second is also a sequence
of $\ZZ$-covering maps and the form of the total spaces suggest that its limit is $\RR^\NN$ as well.
However, it is not difficult to check that there does not exist a map $\RR^\NN\to (S^1)^\NN$ that factors through all terms in the sequence.
Thus, the question is what is the inverse limit of the second sequence
of coverings?

To compute the set of path components of the inverse limit of the
second sequence we compare it with the following sequence over
a polyhedral expansion of $(S^1)^\NN$:
$$\xymatrix{
\RR\times S^1 \ar[d]_{p_1} & \RR\times\RR\times S^1 \ar[d]_{p_2} \ar[l]_{f_2} &
\RR\times\RR\times\RR\times S^1 \ar[d]_{p_3} \ar[l]_{f_3} & \cdots \ar[l] \\
S^1\times S^1  & S^1\times S^1\times S^1  \ar[l] &
S^1\times S^1\times S^1\times S^1  \ar[l] & \cdots \ar[l]
}$$
The maps $f_i\colon \RR^i\times S^1\to \RR^{i-1}\times S^1$ are defined as a product of a projection
on the first $i$-components with the square map (i.e. two-fold covering map) on $S^1$. Furthermore,
the maps $p_i\colon \RR^i\times S^1\to (S^1)^{i+1}$ are covering maps corresponding to respectively the
cyclic subgroup generated by the element $(2^i,2^{i-1},\ldots,2,1)\in \ZZ^{i+1}=\pi_1((S^1)^{i+1})$.
It is easy to check that the diagram commutes and that the $i$-th term in the original sequence is
the pull-back of $p_i$ along the projection map as in the diagram
$$\xymatrix{
\RR^i\times S^1\times \ldots \ar[d] \ar[r] &\RR^i\times S^1\ar[d]^{p_i}\\
(S^1)^{\NN} \ar[r] & (S^1)^{i+1}
}$$
It follows that the inverse limit of the original sequence and the inverse limit of the sequence
of covering over the polyhedral expansion coincide. We may now apply Theorem \ref{thm:expansion lim1}
to determine the set of components of the total space of the limit. In fact, the inverse sequence of the
fundamental groups of the covering spaces is
$$\xymatrix{ \ZZ & \ZZ \ar[l]_2 & \ZZ \ar[l]_2 & \ZZ \ar[l]_2 & \cdots\ar[l] }$$
which is not Mittag-Leffler, therefore its derived inverse limit
is non-trivial (it is actually an
uncountable abelian group). As a consequence, the total space of the
inverse limit has uncountably many
components. Indeed, by a closer examination of the inverse limit of
coverings over the polyhedral expansion of $(S^1)^\NN$ we can conclude
that the total space of the limit in the second case is homeomorphic
to the product of the dyadic solenoid ${\rm Sol}_2$ with $\RR^\NN$.
\end{example}


\section{Algebraic applications}

We begin by describing a construction of inverse sequences of covering
projections whose limits correspond to meaningful subgroups of the
fundamental group. Given a countable CW-complex $X$ and a continuous map $f\colon X\to S^1$
let $p\colon \wX\to X$ be the covering map obtained as a pullback of the universal
covering of $S^1$ along $f$, as in the following diagram
$$\xymatrix{
\wX \ar[r] \ar[d]_p & \RR \ar[d]^e \\
X \ar[r]_f & S^1
}$$
It is easy to check that $p$ is the covering map that corresponds
to the kernel of $f_*$, that is to say
$$\pi_1(\wX)\cong\im (p_*\colon \pi_1(\wX)\to\pi_1(X))=
\Ker(f_*\colon \pi_1(X)\to\pi_1(S^1)).$$
Clearly, $p$ is a regular covering whose fibres can be naturally
identified with the infinite cyclic group $\ZZ$, so we often say
that $p$ is the $\ZZ$-covering, corresponding to the map
$f\colon X\to S^1$ (or rather, to its homotopy class). Since $S^1$ is
an Eilenberg-MacLane space of type $K(\ZZ,1)$, there is a bijection
$$[X,S^1]=\Hom(\pi_1(X),\ZZ).$$
Thus, we may also say that $p$ is the $\ZZ$-covering corresponding
to a given homomorphism $\varphi\colon \pi_1(X)\to\ZZ$, in the sense
that $\pi_1(\wX)=\Ker\varphi$. Note, that every non-trivial subgroup
of $\ZZ$ is isomorphic to $\ZZ$, so we may assume without loss
of generality that $\varphi$ is surjective.

Given a sequence of homomorphisms  $\varphi_1,\varphi_2,\varphi_3,\ldots\colon\pi_1(X)\to\ZZ$,
let
$$\Phi_n:=(\varphi_1,\ldots,\varphi_n)\colon \pi_1(X)\to\ZZ^n$$ and let
$p_n\colon\wX_n\to X$ be the covering projection whose fundamental group is
$$K_n:=\Ker\Phi_n=\Ker \varphi_1\cap\ldots\cap\Ker\varphi_n.$$
Thus we obtain an inverse sequence of coverings
$$\xymatrix{
\wX_1 \ar[d]_{p_1} & \wX_2 \ar[d]_{p_2} \ar[l]_{f_{1}} & \wX_3 \ar[d]_{p_3}
\ar[l]_{f_{2}} & \cdots\ar[l] \\
X \ar@{=}[r] & X \ar@{=}[r] & X \ar@{=}[r] & \cdots }$$
where each $f_n$ is a covering projection. To compute the fibre of $f_n$ note that have
a short exact sequence
$$0\to K_{n-1}/K_n \longrightarrow \pi_1(X)/K_n \longrightarrow \pi_1(X)/K_{n-1}\to 0$$
where the second and third group are subgroups of $\ZZ^n$ and $\ZZ^{n-1}$ respectively.
From this we deduce that $K_{n-1}/K_n$ is either trivial or isomorphic to $\ZZ$, therefore
all $f_n$ are either trivial coverings (identity maps) or $\ZZ$-coverings.


Since $\Hom(\pi_1(X),\ZZ)$ is countable, we may apply the above construction to
the sequence of all homomorphisms from $\pi_1(X)$ to $\ZZ$. The limit of the resulting
inverse sequence of coverings is a lifting projection $\widehat p\colon\widehat X\to X$
whose fundamental group is
$$\pi_1(\widehat X)=\bigcap_{\varphi\colon\pi_1(X)\to\ZZ} \Ker\varphi.$$

Note that if $\pi_1(X)$ is finitely generated (e.g., if the 1-skeleton of $X$ is finite), then
$\widehat p$ is a covering projection. In fact, in that case
$\Hom(\pi_1(X),\ZZ)=\Hom(H_1(X),\ZZ)$ is also finitely generated and the
tower of coverings is actually finite (i.e., all but finitely many coverings in the sequence
are trivial). Alternatively, $\widehat p$ can be obtained as a covering of $X$
that corresponds to the kernel of the natural homomorphism
$\pi_1(X)\to FH_1(X)$, where $FH_1(X)$ denotes the maximal free abelian quotient
of $H_1(X)$.

What is the algebraic meaning of the intersection of kernels of
homomorphisms to some group? It is well-known that the commutator
subgroup $F_n'=[F_n,F_n]$ of the free group on $n$ generators $F_n$ consists
of all words in $F_n$ for which the sum of exponents of each letter
equals 0. This description is not intrinsic,
as it requires to choose a basis for the free group. An equivalent
description without reference to a basis is the following:
if $F$ is a free group (on any set of generators), then
$$F'=\bigcap_{f\colon F\to\ZZ}\Ker(f).$$
This approach was used in Cannon-Conner \cite[Section 4]{Cannon-Conner 1}
to describe the \emph{big commutator subgroup} of the fundamental
group of the Hawaiian earring $\HEG=\pi_1(H)$:
$$BC(\HEG)=\bigcap_{f\colon \HEG\to\ZZ}\Ker(f).$$
The big commutator subgroup is much larger than the usual
commutator subgroup $\HEG'$ but shares many interesting
properties  with the latter.

One can consider intersections of homomorphisms to other groups as well,
for example finite or torsion-free groups. In
\cite{Conner-Kent-Herfort-Pavesic} we gave an intrinsic description
of he second commutator subgroup $F_n''$ using homomorphisms
to the Baumslag-Solitar group $B(1,n)$.

Let us introduce the following functor for arbitrary groups $G$
and $A$:
$$\Ker_A(G):=\bigcap_{f\colon G\to A}\Ker(f).$$
Note that $\Ker_A(G)$ is a fully characteristic subgroup of $G$.
More generally, if $\varphi\colon G\to H$ is a homomorphism and if
$x\in \Ker_A(G)$, then for every $\psi\colon H\to A$ we have that
$\psi\circ\varphi(x)=0$, therefore $\varphi(x)\in\Ker_A(H)$.
In categorical terms the correspondence $G\mapsto \Ker_A(G)$ is a
covariant functor. It turns out that many characteristic subgroups
can be described as intersection of kernels to some group $A$.
In fact we may view $\Ker_A(G)$ as the part of $G$ that cannot be
represented in a product of copies of the group $A$. For example
$\Ker_\ZZ(G)=0$ if, and only if, $G$ is residually free-abelian.

We mentioned before that $\Ker_\ZZ(F)=F'$ and $\Ker_{B(1,n)}(F)=F''$
for every free group $F$.
Recall that a torsion-free abelian group $A$ is \emph{slender}
if every homomorphism
$\varphi\colon \ZZ^\NN\to A$ can be factored through some finite rank
free abelian group, i.e., there exists a homomorphism
$\varphi'\colon \ZZ^n\to A$ so that the following diagram commutes
$$\xymatrix{
\ZZ^\NN \ar[rr]^\varphi \ar[dr]_{\mathrm{pr}_n} & & A\\
& \ZZ^n \ar[ur]_{\varphi'}}
$$
(see Fuchs \cite[Ch. XIII]{Fuchs}) where $\mathrm{pr}_n$ is the projection map onto the first $n$ factors of $\ZZ^\NN$. Free abelian groups are slender, subgroups and extensions of slender groups are also slender. A reduced abelian group is slender if, and only if, it does not contain a subgroup isomorphic to the group $\ZZ^\NN$ or to a group of $p$-adic integers for some prime $p$ (cf. Fuchs \cite{Fuchs}).

Katsuya Eda \cite{eda92}
extended this concept to arbitrary groups by defining a group $A$ to be \emph{non-commutatively slender} (nc-slender) if every homomorphism $\varphi\colon \HEG\to A$ from the Hawaiian earring group $\HEG$ to $A$ can be factored through some free (non-commutative) group of finite rank as in $$\xymatrix{
\HEG \ar[rr]^\varphi \ar[dr]_{\mathrm{pr}_n} & & A\\
& F_n \ar[ur]_{\varphi'}}
$$
where $\mathrm{pr}_n$ is the homomorphism induced by collapsing all but the first $n$ loops of the Hawaiian earring.  

\begin{lemma}
\label{lem:lim slender}
If $A$ is slender, then
$$\Ker_A(\ZZ^\NN)=\varprojlim \Ker_A(\ZZ^n).$$
Similarly, if $A$ is nc-slender, then
$$\Ker_A(\hF)=\varprojlim \Ker_A(F_n) \ \ \text{and} \ \
\Ker_A(\HEG)=\HEG\cap \varprojlim \Ker_A(F_n).$$
\end{lemma}
\begin{proof}
The natural projections $\mathrm{pr}_n\colon\ZZ^\NN\to\ZZ^n$ induce an isomorphism $\ZZ^\NN\to\varprojlim \ZZ^n$. By functoriality of $\Ker_A$, the homomorphisms $\Ker_A(\ZZ^\NN)\to \Ker_A(\ZZ^n)$ induced by $\mathrm{pr}_n$ are coherent and induce a homomorphism $\Ker_A(\ZZ^\NN)\to\varprojlim \Ker_A(\ZZ^n).$
This homomorphism is injective, because it is a restriction of the isomorphism $\ZZ^\NN\to\varprojlim \ZZ^n$. To prove surjectivity, consider an element $(x_i)\in \varprojlim \Ker_A(\ZZ^n)=\ZZ^\NN$.  We need only show that $(x_i)\in \Ker_A(\ZZ^\NN)$. Fix a homorphism $\varphi\colon\ZZ^\NN\to A$. Since $A$ is slender, $\varphi$ can be factored as $\varphi=\varphi'\circ\mathrm{pr}_n$ for some $n$ and some $\varphi'\colon \ZZ^n\to A$. Then $\varphi((x_i))=\varphi'(x_n)$ which is trivial since $x_n \in \Ker_A(\ZZ^n)$.  Thus $\Ker_A(\ZZ^\NN)\to\varprojlim \Ker_A(\ZZ^n)$ is surjective
as well.


In \cite{homlimpreprint}, the authors show that if $A$ is nc-slender then every homomorphism from $\hF$ to $A$ also factors through a projection to $F_n$ for some $n$. Then the proof for the non-commutatively slender cases follows analogously.
\end{proof}

As a consequence, $A$-kernels of $\hF$ with respect to a slender group $A$
give rise to inverse limits of coverings over the Hawaiian earring.
Indeed, recall that the Hawaiian earring $H$ can be represented as inverse limit of
a sequence
$$\cdots \leftarrow X_n\leftarrow X_{n+1} \leftarrow\cdots\leftarrow H,$$
where $X_n$ is a wedge of $n$ circles and the bonding maps are the
obvious projections. For each $n$ the fundamental group of $X_n$ is $F_n$,
the free group on $n$ generators and there is the universal covering projection
$q_n\colon \hX_n\to X_n=\hX_n/F_n$. The limit of the resulting inverse sequence of coverings
is a  fibration with unique path lifting property
$$q\colon \widehat{H}=\varprojlim \hX_n\to H=\varprojlim X_n.$$
By \cite[Section 4]{Conner-Pavesic} the group $\hF$ acts freely and transitively on
the fibres of $q$ so we may consider
the quotient of $\hH$ with respect to any (normal) subgroup of $\hF$.

\begin{proposition}
\label{prop:fibration}
If $A$ is any nc-slender group, then  the projection
$$q_A\colon\hH/\Ker_A(\hF)\to H$$
can be obtained as the inverse limit of coverings $q_n'\colon \hX_n/\Ker_A(F_n)\to X_n$,
and is
therefore a fibration with unique path-lifting property.
\end{proposition}
\begin{proof}
For each $n$ we have a commutative diagram of the form
\[\xymatrix{
\hH \ar[r] \ar[d] & \hX_n\ar[d] \\
\hH/\Ker_A(\hF) \ar[r] \ar[d]_{q_A} & \hX_n/\Ker_A(F_n)\ar[d]^{q_n'} \\
H \ar[r] & X_n
}\]
where for each $n$ the map $q_n'$ is a covering projection with fibre
$F_n/\Ker_A(F_n)$. These coverings
form an inverse sequence and by naturality we get a mapping to the inverse limit
\[\xymatrix{
\hH/\Ker_A(\hF) \ar[r]^-f \ar[d]_{q_N} & \varprojlim \big(\hX_n/\Ker_A(F_n)\big)\ar[d]^q \\
H \ar@{=}[r] & \varprojlim X_n
}\]

To prove our claim it is sufficient to show that the restriction of $f$ to the fibres
\[f\colon \hF/\Ker_A(\hF)\to \varprojlim \big(F_n/\Ker_A(F_n)\big)\]
is an isomorphism. Consider the following inverse sequence of
short exact sequences
\[\xymatrix{
1 \ar[r] & \Ker_A(F_1) \ar[r] & F_1 \ar[r] & F_1/\Ker_A(F_1) \ar[r] &  1\\
1 \ar[r] & \Ker_A(F_2) \ar[u] \ar[r] & F_2\ar[u] \ar[r] & F_2/\Ker_A(F_2)
\ar[r] \ar[u]&  1\\
1 \ar[r] & \Ker_A(F_3) \ar[u] \ar[r] & F_3\ar[u] \ar[r] & F_3/\Ker_A(F_3)
\ar[r] \ar[u]&  1\\
 & \vdots \ar[u] & \vdots \ar[u] & \vdots \ar[u]
}\]
Each projection $F_{n+1}\to F_n$ is a split surjection, so by functoriality of $\Ker_A$
all bonding homomorphisms in the first
column are surjective. By \cite[Section 11.3]{Geoghegan}  the resulting sequence of
inverse limits is also exact.
Moreover, by Lemma \ref{lem:lim slender} $\varprojlim \Ker_A(F_n)=\Ker_A{\hF}$ so we have
the following short exact sequence
$$\xymatrix{
1 \ar[r] & \Ker_A(\hF) \ar[r] & \hF \ar[r] & \varprojlim \big(F_n/\Ker_A(F_n)\big) \ar[r] &
1,\\
}$$
which implies that the projections $\hF\to F_n$ induce the isomorphism
$$\hF/\Ker_A(\hF)\cong \varprojlim \big(F_n/\Ker_A(F_n)\big)$$
as claimed.
\end{proof}

In order to prove the main results of this section we will need an algebraic
lemma that is reminiscent of Theorem \ref{thm:expansion lim1}. Let $G$ be a subgroup
of the inverse limit
of a sequence of groups $\{G_i\}$ such that all projections
$G\to G_i$ are surjective. Furthermore, let $H_i$ be a decreasing
sequence of subgroups of $G$ and for each $i$ let $\overline H_i$ be
the image of $H_i$ in $G_i$. Observe that $G/H_i\cong G_i/\overline H_i$
(a bijection for arbitrary groups $H_i$, and a group isomorphism
if $H_i$ are normal subgroups of $G$).
We can fit the above data in a commutative diagram with
exact rows:
$$\xymatrix{
1 \ar[r] & \bigcap H_i \ar[r]\ar[d]  & G \ar[r]\ar[d]& \varprojlim G/H_i \ar[r]\ar@{=}[d]  &  \lim{^1} H_i \ar[r]\ar[d] & \bullet\\
1 \ar[r] & \varprojlim \overline{H}_i \ar[r] & \varprojlim G_i \ar[r] & \varprojlim G_i/\overline{H}_i \ar[r]  &  \lim{^1} \overline{H}_i \ar[r] & \bullet}
$$
The following lemma is proved by elementary diagram chasing.

\begin{lemma}
\label{lem:alg lim1}
$\Ker(\bigcap H_i \to \varprojlim \overline{H}_i)=
\Ker(G\to\varprojlim G_i)$, thus we have an exact sequence
$$1\longrightarrow\Ker(G\to\varprojlim G_i)\longrightarrow \bigcap H_i
\longrightarrow \varprojlim \overline{H}_i$$
Moreover
$\varprojlim{^1} H_i\to\varprojlim{^1}\overline{H}_i$ is a surjection of homogeneous
sets whose fibres can be naturally identified with the cokernel
of the homomorphism
$$G\to (\varprojlim G_i)/(\varprojlim \overline{H}_i)$$

As a consequence, $\varprojlim{^1}H_i$ is trivial iff $\varprojlim{^1}
\overline H_i$ is trivial
and $G\to (\varprojlim G_i)/(\varprojlim \overline{H}_i)$ is surjective.
\end{lemma}

We have previously considered the inverse sequence of wedges of circles
$$\cdots \leftarrow X_n\leftarrow X_{n+1} \leftarrow\cdots\leftarrow H,$$
converging to the Hawaiian earring. Since $\pi_1(X_n)=F_n$, the free group
on $n$ generators $F_n$, the inverse sequence of spaces gives rise to
a homomorphism from the Hawaiian earring group $\HEG:=\pi_1(H)$ to the
inverse limit of finite rank free groups $\hF:=\varprojlim F_n$. By a result of Higman \cite{Higman}, see also
\cite{Cannon-Conner 1} and \cite{EdaKawamura}, the homomorphism
$\HEG\to\hF$ is injective, so the Hawaiian earring group can be
viewed as a subgroup of $\hF$. It is known that both the group $\HEG$
and its index in $\hF$ are uncountable, and that $\HEG$ can be viewed
as a dense subgroup of $\hF$ with respect to the inverse limit topology
on $\hF$. The following theorem gives a more precise description of the
relation between the two groups. The proof is a combination of algebraic
and geometric techniques developed in the previous sections.

\begin{theorem}
\label{thm:HEG internal}
The group $\hF$ is equal to the internal product of its subgroups
$\HEG$ and $\Ker_\ZZ(\hF)$, i.e., $\hF=\HEG\cdot \Ker_\ZZ(\hF)$.
\end{theorem}
\begin{proof}
We are going to apply Lemma \ref{lem:alg lim1} to the following data: let $G_i$ be the free group on $i$ generators $F_i$, $G=\HEG$, and $H_i$ the kernel of the epimorphism $\HEG\to \pi_1(X_i)\to H_1(X_i)\cong\ZZ^i$.  Then $\varprojlim G_i=\hF$.  Observe that the image of $H_i$ in $\pi_1(X_i) =F_i$ is precisely the commutator subgroup $F_i'=\Ker_\ZZ(F_i)$, since $\HEG$ maps surjectively onto $\pi_1(X_i)$. The bonding maps in the inverse sequence of groups $\{\overline H_i\}$ are surjective, therefore $\varprojlim^1\overline H_i$ is trivial. Moreover, by Lemma \ref{lem:lim slender}
$\varprojlim\overline H_i=\varprojlim\Ker_\ZZ(F_i)=\Ker_\ZZ(\hF)$. Thus, by Lemma \ref{lem:alg lim1} the formula $\hF=\HEG\cdot \Ker_\ZZ(\hF)$ holds if, and only if, the derived inverse limit $\varprojlim^1 H_i$ is trivial.

It is well-known that the limit $\varprojlim^1 H_i$ is precisely the set
of path-components of the inverse limit of covering maps
that correspond to
the sequence of kernels of the homomorphisms $\{\HEG\to H_1(X_i)\}$ (see \cite{Cohen70, Hirschorn15}, as well as our discussion in Section \ref{sec:Preliminaries on inverse limits of spaces and groups}).
Analogously as in the examples in Section 3, we can represent this
lifting projection as a limit of covering spaces over the approximations
of the Hawaiian earrings by finite wedges of circles $X_i$. The covering
space over $X_i$ corresponding to the commutator subgroup $F_i$ can be
identified with the integral 1-dimensional grid in $\RR^i$, i.e.
$$\wX_i =\{(x_1,\ldots,x_i)\in\RR^i\mid x_j\notin\ZZ\ \ \text{for at most
one index}\ j\}.$$
Observe that the bonding maps in the inverse system
$$\wX_1\longleftarrow \wX_2\longleftarrow \wX_3 \longleftarrow\cdots$$
are retractions, and so their inverse limit $\wX$ is the
1-dimensional integral grid in $\RR^\NN$. Alternatively, we may describe
$\wX$ by the following pullback diagram
$$\xymatrix{
\wX \ar[r] \ar[d] & \RR^\NN \ar[d]^{e^\NN}\\
H \ar@{^(->}[r] & (S^1)^\NN
}$$
Clearly, $\wX$ is path-connected, which completes the proof of our claim.
\end{proof}

Note that there one could also consider the commutator subgroup $\hF'$ of
$\hF$, which is however much smaller than $\varprojlim F_n'$. In fact, $\hF>\HEG\cdot \hF'$.

Continuing the previous line of thought we may ask whether $\hF$ can be obtained
by adding some other term of the derived series of $\hF$ to the group $\HEG$. As before, we replace
$\hF''$ with a suitable inverse limit group. We already mentioned that
the second derived group of a free group $F$ can be described as a kernel, $F''=\Ker_{B}(F)$
(see \cite[Theorem 1]{Conner-Kent-Herfort-Pavesic}) for any solvable, deficiency $1$ group $B$ that is
not virtually abelian.  By work of Wilson \cite{Wilson96}, any solvable deficiency $1$ group is isomorphic to
a Baumslag-Solitar group $B(1,m)$ for some $m$. Since $B(1,m)$
is nc-slender (see \cite{ConnerCorson19}), Lemma \ref{lem:lim slender} implies
that $\Ker_{B}(\hF)=\varprojlim \Ker_{B}(F_n)$ for every group $B$ that is solvable, of deficiency $1$ and
is not virtually abelian.  As well, we have
$$\hF''\le \Ker_{B}(\hF)\le \Ker_{\ZZ}(\hF).$$

In the next theorem we reverse the reasoning and use our methods to show
that this subgroup of $\hF$ is too small to generate, together with
the Hawaiian earring group the entire group $\hF$.

\begin{theorem}
\label{thm:HEG}
$\hF\ne\HEG\cdot \Ker_{B}(\hF)$ for every group $B$ that is solvable, of deficiency $1$ and
is not virtually abelian.
\end{theorem}

\begin{proof}
Let $B$ be a solvable deficiency $1$ group that is not virtually abelian.  Then $B$ is isomorphic to a Baumslag-Solitar group $B(1,m)$ and, for a
free group of rank $n$, we have
$\Ker_B(F_n)=F_n''=\Ker_\ZZ(\Ker_\ZZ(F_n))$.
We may repeat almost verbatim the algebraic part of the proof of the previous
theorem and obtain that
$\hF=\HEG\cdot \Ker_{B}(\hF)$ if, and only if, the inverse limit
$\widetilde Y$ of the sequence of coverings over the
Hawaiian earring $H$, determined
by the kernels
of homomorphisms $\HEG\to F_i/F_i''$, is path-connected.
Thus, in order to
prove our claim, we must show that $\widetilde Y$ is not path-connected.

We will use the same notation as in the proof of the previous
theorem.
Let $\wX$ be the integral grid in $\RR^\NN$ and let $p\colon \wX\to H$ be
the lifting projection obtained as the limit of the sequence
$$\xymatrix{
\wX_1 \ar[d]_{p_1} & \wX_2 \ar[d]_{p_2} \ar[l] & \wX_3
\ar[d]_{p_3} \ar[l] & \cdots\ar[l] & \wX \ar[d]^p\ar[l]\\
X_1  & X_2 \ar[l] & X_3 \ar[l] & \cdots\ar[l] & H\ar[l]}$$
where $\pi_1(\wX_n)=F_n'$. Let $\wY_n$ be a cover of
$\wX_n$ corresponding to the subgroup $\bigl(\pi_1(\wX_n)\bigr)'=F_n''$
and let $\wY$ be the resulting inverse limit.
We will construct a continuous surjection from $\wY$ to a path-disconnected space.
Since $\pi_1(\wX_2)=F_2'$ is a countably generated free group, we can find
a sequence of homomorphisms $f_i:\pi_1(\wX_2)\to \ZZ$ such that
$$\bigcap\limits_{i=1}^n \Ker(f_i)\subsetneq \bigcap\limits_{i=1}^{n-1} \Ker(f_i)\ \ \text{and}
\ \ \bigcap\limits_{i=1}^\infty \Ker(f_i)=F_2''.$$
Let $\widetilde Z_n'$ be the cover of $\wX_2$ corresponding to the subgroup
$\bigcap\limits_{i=1}^n \Ker\bigl(\pi_1(\wX_2) \stackrel{f_i}{\longrightarrow} \ZZ\bigr).$
The covers $\widetilde Z_n'$ form an inverse sequence, whose limit
$\widetilde Z':=\varprojlim \widetilde Z_n'$ is path-disconnected by Corollary \ref{cor:cov over CW}.
As a consequence, the pullback of $\widetilde Z'$ along the projection $p\colon\wX\to\wX_2$
$$\xymatrix{
\widetilde Z' \ar[d] & \widetilde Z \ar[l] \ar[d]\\
\wX_2 & \wX \ar[l]^p
}$$
yields a lifting
projection $\widetilde Z\to \wX$, whose total space is also path-disconnected. Note that
$\widetilde Z$ can be alternatively obtained as inverse limit of a sequence
$\widetilde Z=\varprojlim \widetilde Z_n$, where $\widetilde Z_n$ is the pull-back of
$\widetilde Z_n'$ along the projection $p_n\colon\wX_n\to\wX_2$. Since commutator subgroups
are characteristic we have $p_{n*}(\pi_1(\wY_n))\le \pi_1(\wY_2)\le \pi_1(\widetilde Z')$, thus
$\pi_1(\wY_n)\le\pi_1(\widetilde Z_n)=p_{n*}^{-1}(\pi_1(\widetilde Z'))$. It follows
that   $\wY_n$ covers $\widetilde Z_n$ for every $n$ and we obtain the
 following diagram:
$$\xymatrix{
\wY_2 \ar[d] & \wY_3 \ar[d] \ar[l] & \cdots \ar[l] & \wY_{n-1}
\ar[d] \ar[l] & \wY_n \ar[d] \ar[l] & \wY_{n+1} \ar[d] \ar[l] & \cdots \ar[l] &\wY\ar[d]\ar[l]\\
\widetilde Z_2 \ar[d] & \widetilde Z_3 \ar[d] \ar[l] & \cdots \ar[l] & \widetilde Z_{n-1} \ar[d]
\ar[l] & \widetilde Z_n \ar[d] \ar[l] & \widetilde Z_{n+1} \ar[d] \ar[l] & \cdots\ar[l] &
\widetilde Z\ar[d]\ar[l]\\
\wX_2  & \wX_3 \ar[l] & \cdots \ar[l] & \wX_{n-1} \ar[l] &
\wX_{n} \ar[l] & \wX_{n+1} \ar[l] & \cdots \ar[l] & \ar[l]\wX}
$$
By construction, the fibre of the lifting projection $\wY\to\wX$ maps surjectively to the fibre of
$\widetilde Z\to\wX$.
Since every point in $\widetilde Z$ (respectively $\wY$) is connected by a path to a point in
the fibre, it follows, that the projection $\wY\to \widetilde Z$ is surjective, therefore
$\wY$ is not path-connected.
\end{proof}



\end{document}